\documentclass[preprint,3p,times]{elsarticle}

\usepackage{amssymb}
\usepackage{amsmath, amsthm, amssymb, amsfonts}
\usepackage{enumerate}
\usepackage[shortlabels]{enumitem}

\newtheorem{theorem}{Theorem}[section]

\newtheorem{lemma}[theorem]{Lemma}

\theoremstyle{definition}
\newtheorem{definition}[theorem]{Definition}

\journal{Nuclear Physics B}

\begin{document}

\begin{frontmatter}



\title{Dynamics of an Intra-host Diffusive Pathogen Infection Model}


\author{Shohel Ahmed} 

\affiliation{organization={Department of Mathematical and Statistical Sciences},
            addressline={University of Alberta}, 
            city={Edmonton},
            postcode={T6L 3P8}, 
            state={Alberta},
            country={Canada}}

\begin{abstract}
In this paper, we first propose a diffusive pathogen infection model with a general incidence rate that incorporates cell-to-cell transmission. Using the theory of monotone dynamical systems, we prove that the model exhibits global threshold dynamics characterized by the basic reproduction number ($\mathcal{R}_{0}$), which is defined as the spectral radius of the next generation operator. We then derive a discrete counterpart of the continuous model by applying a nonstandard finite difference scheme. The results show that the discrete model preserves the positivity and boundedness of solutions, which guarantees the well-posedness of the problem, and that this scheme also preserves all equilibria of the original continuous model. By constructing suitable Lyapunov functionals for both models, we further show that the global threshold dynamics is completely determined by the basic reproduction number. In addition, through sensitivity analysis we identify the most influential parameters that effectively alter the disease dynamics. Finally, we illustrate the theoretical results with an example and numerical simulations, which extend and generalize some known results.
\end{abstract}




\end{frontmatter}



\section{Introduction}
Over the past few decades, considerable effort has been devoted to the mathematical modelling of within-host pathogen infection. Such models have been used to describe the dynamics of various infectious diseases inside a host, including HIV, HCV, HBV and HTLV, among others. The classical within-host virus model consists of a system of three ordinary differential equations \cite{r1, r2}, where it is assumed that cells and virions are well mixed, so the spatial movement of free virus is ignored. To investigate the effects of spatial structure on virus dynamics, Wang and Wang \cite{r3} proposed the following diffusive system, in which the motion of the virus is assumed to follow the Fickian diffusion \cite{r4}:
\begin{equation}\label{ee1}
\begin{cases}
\begin{aligned}
    \dfrac{\partial S(x,t)}{\partial t} &= \Lambda - \beta_1 S V - d_S S,\\[4pt]
    \dfrac{\partial I(x,t)}{\partial t} &= \beta_1 S V - d_I I, \\[4pt]
    \dfrac{\partial V(x,t)}{\partial t} &= D_3 \Delta V + \alpha I - d_V V,
\end{aligned}
\end{cases}
\end{equation}
where $S(x,t)$, $I(x,t)$ and $V(x,t)$ denote, respectively, the densities of susceptible (uninfected) cells, infected cells, and free virus at position $x$ and time $t$. Susceptible cells are produced at a constant rate $\Lambda$ and become infected by free virions at rate $\beta_1 S V$. The parameters $d_S$, $d_I$ and $d_V$ denote the death rates of uninfected cells, infected cells and free virus, respectively. Infected cells produce free virions at rate $\alpha I$. The constant $D_3$ is the diffusion coefficient of the virus and $\Delta$ denotes the Laplacian operator.

Notice that the above system \eqref{ee1} only focuses on virus-to-cell spread in the bloodstream, although several studies have shown that cell-to-cell transmission, that is, direct contact between an infected source cell and a susceptible target cell, plays a crucial role in viral spread in vivo \cite{r5, r6, r7, r8}. A better understanding of viral cell-to-cell spread can therefore improve our ability to intervene in efficient viral transmission. For more detailed work on target cell dynamics and cell-to-cell transmission, one may consult \cite{r9, r10, r11, r12, r13, r14, r15, r16, r17} and the references therein. On the other hand, the bilinear incidence used in \eqref{ee1} is a simple description of infection. As pointed out in \cite{r14, r18}, a general incidence function can provide a unified framework by omitting inessential details. Motivated by these considerations, we propose the following pathogen infection model on the domain $Q = \mathbb{R}^{+} \times \Omega$,
\begin{equation}\label{e1}
\begin{cases}
\begin{aligned}
    \dfrac{\partial S(x,t)}{\partial t} &= D_1 \Delta S + \Lambda - S f(V) - S g(I) - d_S S, && x \in \Omega,\ t>0,\\[4pt]
    \dfrac{\partial I(x,t)}{\partial t} &= D_2 \Delta I + S f(V) + S g(I) - (\gamma + d_I) I, && x \in \Omega,\ t>0,\\[4pt]
    \dfrac{\partial V(x,t)}{\partial t} &= D_3 \Delta V + \alpha I - d_V V, && x \in \Omega,\ t>0,
\end{aligned}
\end{cases}
\end{equation}
where $D_1$ and $D_2$ are the diffusion coefficients of susceptible and infected cells, respectively \cite{r19}, and $\gamma$ is the lysis rate of infected cells \cite{r20}. The infection terms are assumed to be nonlinear responses to the densities of free virus and infected cells, and take the forms $S f(V)$ and $S g(I)$, where $f(V)$ and $g(I)$ denote the forces of infection by virus particles and infected cells. We assume that $f$ and $g$ satisfy the following conditions \cite{r21}:
\begin{enumerate}[start=1,label={(\bfseries A\arabic*)}]
\item $f(0)=g(0)=0$ and $f(V),\, g(I) > 0$ for $V,\, I > 0$;
\item $f'(V),\, g'(I) < 0$ and $f''(V),\, g''(I) \le 0$ for $V,\, I \ge 0$.
\end{enumerate}
From (A1) and (A2), the Mean Value Theorem (MVT) implies that
\begin{equation}\label{ee11}
    f'(V)\,V \le f(V) \le f'(0)\,V, \qquad
    g'(I)\,I \le g(I) \le g'(0)\,I, \qquad \text{for } V,\, I \ge 0.
\end{equation}
Biologically, assumptions (A1) and (A2) mean that
(i) infection cannot occur in the absence of virus or infected cells,
(ii) the incidences $S f(V)$ and $S g(I)$ increase as the densities of virus and infected cells increase, and
(iii) the per capita infection rates decrease when the densities become high, due to inhibition or saturation effects, since \eqref{ee11} implies
\[
\bigg(\dfrac{f(V)}{V}\bigg)' \le 0 \quad \text{and} \quad \bigg(\dfrac{g(I)}{I}\bigg)' \le 0.
\]
Clearly, this class of incidence functions includes the usual bilinear and saturated incidences, for example
\[
f(V) = \beta_1 V \quad \text{or} \quad f(V) = \dfrac{\beta_1 V}{1+V}, 
\qquad
g(I) = \beta_2 I \quad \text{or} \quad g(I) = \dfrac{\beta_2 I}{1+I},
\]
where $\beta_1, \beta_2 > 0$ are infection rates.

We consider \eqref{e1} with the initial conditions
\begin{equation}\label{e2}
    S(x,0) = \varphi_1(x) \ge 0,\quad
    I(x,0) = \varphi_2(x) \ge 0,\quad
    V(x,0) = \varphi_3(x) \ge 0,\quad x \in \Omega,
\end{equation}
where $\varphi_1, \varphi_2, \varphi_3 \in C^{2}(\Omega) \cap C^{0}(\overline{\Omega})$, together with the homogeneous Neumann boundary conditions
\begin{equation}\label{e3}
    \dfrac{\partial S}{\partial \nu} = \dfrac{\partial I}{\partial \nu} = \dfrac{\partial V}{\partial \nu} = 0, \quad x \in \partial \Omega,\ t>0,
\end{equation}
where $\Omega \subset \mathbb{R}^{n}$ is a bounded domain with piecewise smooth boundary $\partial \Omega$ and $\nu$ denotes the unit outward normal vector to $\partial \Omega$.

In general, an exact solution of the continuous system such as \eqref{e1} is difficult, and in many cases impossible, to obtain. It is therefore natural to require that a numerical method preserve, at the discrete level, the main qualitative properties of the corresponding continuous system. However, selecting a discrete scheme that efficiently preserves the global dynamics of the original continuous model is still a challenging problem \cite{r22}. To address this, Mickens proposed a robust nonstandard finite difference (NSFD) methodology \citep{r23}, which has since been widely applied to many dynamical models \citep{r24, r25, r26, r27, r28, r29}. Motivated by \cite{r23}, we apply the NSFD approach to the continuous system \eqref{e1} and obtain the following discrete scheme:
\begin{equation}\label{a1}
\begin{cases}
\begin{aligned}
\dfrac{S_{n}^{k+1} - S_{n}^{k}}{\Delta t} &= D_1 \dfrac{S_{n+1}^{k+1} - 2S_{n}^{k+1} + S_{n-1}^{k+1}}{(\Delta x)^{2}} + \Lambda - S_{n}^{k+1} f(V_{n}^{k}) - S_{n}^{k+1} g(I_{n}^{k}) - d_S S_{n}^{k+1}, \\[4pt]
\dfrac{I_{n}^{k+1} - I_{n}^{k}}{\Delta t} &= D_2 \dfrac{I_{n+1}^{k+1} - 2I_{n}^{k+1} + I_{n-1}^{k+1}}{(\Delta x)^{2}} + S_{n}^{k+1} f(V_{n}^{k}) + S_{n}^{k+1} g(I_{n}^{k}) - (\gamma + d_I) I_{n}^{k+1}, \\[4pt]
\dfrac{V_{n}^{k+1} - V_{n}^{k}}{\Delta t} &= D_3 \dfrac{V_{n+1}^{k+1} - 2V_{n}^{k+1} + V_{n-1}^{k+1}}{(\Delta x)^{2}} + \alpha I_{n}^{k+1} - d_V V_{n}^{k+1},
\end{aligned}
\end{cases}
\end{equation}
where the spatial domain is $x \in \Omega = [a,b]$, $a,b \in \mathbb{R}$, and $\Delta x = (b-a)/M$ is the spatial step size that divides the interval into $M$ equal subintervals. The time step is denoted by $\Delta t$. We denote by $S_{n}^{k}$, $I_{n}^{k}$ and $V_{n}^{k}$ the approximations of $S(x_{n}, t_{k})$, $I(x_{n}, t_{k})$ and $V(x_{n}, t_{k})$, respectively, at the mesh points $(x_{n}, t_{k})$, where
\[
x_{n} = a + n \Delta x,\quad n = 0,1,\dots,M, \qquad t_{k} = k \Delta t,\quad k \in \mathbb{N}.
\]
The discrete initial conditions are
\begin{equation}\label{a2}
S_{n}^{0} = \varphi_{1}(x_{n}) > 0,\quad I_{n}^{0} = \varphi_{2}(x_{n}) > 0,\quad V_{n}^{0} = \varphi_{3}(x_{n}) > 0,
\end{equation}
and the homogeneous Neumann boundary conditions are incorporated as
\begin{equation}\label{a3}
S_{-1}^{k} = S_{0}^{k},\quad S_{M}^{k} = S_{M+1}^{k},\quad
I_{-1}^{k} = I_{0}^{k},\quad I_{M}^{k} = I_{M+1}^{k},\quad
V_{-1}^{k} = V_{0}^{k},\quad V_{M}^{k} = V_{M+1}^{k},
\end{equation}
for all $k \ge 0$.

Our goal is to show that the NSFD scheme \eqref{a1} preserves the essential qualitative features of the continuous model \eqref{e1}, in particular the positivity, boundedness and global asymptotic stability of the equilibria. The rest of the paper is organized as follows. In Section~\ref{sec2}, we study the dynamical behavior of the continuous system \eqref{e1}, including the existence and uniqueness of positive solutions, the existence of equilibria, the basic reproduction number, and local and global stability results. In Section~\ref{sec3}, we analyze the global dynamics of the discrete system \eqref{a1}. In Section~\ref{sec4}, we present numerical simulations to confirm and illustrate the theoretical findings. A brief conclusion is given at the end of the paper.

\section{Dynamical behavior of the continuous model}\label{sec2}
\subsection{Existence, Uniqueness and Positivity}
To discuss the dynamical behavior of the continuous system (\ref{e1}), first we give the definition of upper and lower solution.
\begin{definition}
Let $(\hat{S}, \hat{I}, \hat{V})$ and $(\check{S}, \check{I}, \check{V})$ in $C(\bar{\Omega} \times [0, \infty)) \cap C^{1,2}(\Omega \times [0, \infty))$ are a pair of upper and lower solution to the problem (\ref{e1}), if $\check{S} \leq \hat{S},\check{I} \leq \hat{I},\check{V} \leq \hat{V}$ in $\bar{\Omega} \times [0, \infty)$ and the following differential inequalities hold:
\end{definition}
\begin{align*}
    \dfrac{\partial \hat{S}(x,t)}{\partial t}\geq&\ D_1\Delta \hat{S}  + \Lambda  - \hat{S} f(\check{V}) - \hat{S}g(\check{I}) - d_S \hat{S},\\
\dfrac{\partial \hat{I}(x,t)}{\partial t}\geq&\ D_2\Delta \hat{I} + \hat{S} f(\hat{V}) +  \hat{S}g(\hat{I})  - (\gamma+d_I) \hat{I},\\
  \dfrac{\partial \hat{V}(x,t)}{\partial t}\geq&\  D_3\Delta \hat{V} + \alpha \hat{I} -  d_V  \hat{V},\\
    \dfrac{\partial \check{S}(x,t)}{\partial t}\leq&\ D_1\Delta \check{S}  + \Lambda  - \check{S} f(\hat{V}) - \check{S}g(\hat{I}) - d_S \check{S},\\
\dfrac{\partial \check{I}(x,t)}{\partial t}\leq&\ D_2\Delta \check{I} + \check{S} f(\check{V}) +  \check{S}g(\check{I})  - (\gamma+d_I) \check{I},\\
  \dfrac{\partial \check{V}(x,t)}{\partial t}\leq&\  D_3\Delta \check{V} + \alpha \check{I} -  d_V  \check{V},
\end{align*}
for $(x,t) \in \Omega \times (0, \infty)$ and
\begin{align*}
&\dfrac{\partial \check{S}}{\partial\nu}\leq 0 \leq \dfrac{\partial \hat{S}}{\partial\nu},\quad
\dfrac{\partial \check{I}}{\partial\nu}\leq 0 \leq \dfrac{\partial \hat{I}}{\partial\nu},\quad
\dfrac{\partial \check{V}}{\partial\nu}\leq 0 \leq \dfrac{\partial \hat{V}}{\partial\nu} ,\quad &(x, t) \in \partial\Omega \times (0, \infty),\\
&\check{S}(x,t)\leq \varphi_1(x,t) \leq \hat{S}(x,t), \quad \check{I}(x,t)\leq \varphi_2(x,t) \leq \hat{I}(x,t),\\
&\check{V}(x,t)\leq \varphi_3(x,t) \leq \hat{V}(x,t),\quad &(x, t) \in \bar{\Omega} \times (0, \infty).
\end{align*}
It is easy to see that $\mathbf{0}=(0, 0, 0)$ and $\mathbf{K}=(K_1, K_2, K_3)$ are a pair of coupled lower-upper solutions to problem (\ref{e1}), where
\begin{align*}
&M_1=\text{max}\bigg\{ \dfrac{\Lambda}{d},  \| \varphi_1\|_{C(\bar{\Omega},\mathbb{R})}   \bigg\}, \quad M_2=\text{max}\bigg\{ \dfrac{\Lambda}{d},  \| \varphi_2\|_{C(\bar{\Omega},\mathbb{R})}   \bigg\}, \quad M_3=\text{max}\bigg\{ \dfrac{\alpha\Lambda}{d_V d},  \| \varphi_1\|_{C(\bar{\Omega},\mathbb{R})}   \bigg\},
\end{align*}
and $d=\text{min}\{d_S, d_I\}$. Using the following lemma provided by Redinger \cite{r30}, we get the existence and uniqueness of the solution.
\begin{lemma}
Let $\hat{U}$ and $\check{U}$ be a pair of coupled upper and lower solutions for problem (\ref{e1}) and suppose that the initial functions $\varphi_i, (i=1, 2, 3)$ are H\"{o}lder continuous in $\bar{\Omega}$. Then problem (\ref{e1}) has exactly one regular solution $U(x,t)= (S(x,t), I(x,t), V(x,t))$ satisfying $\check{U}\leq U \leq \hat{U}$ in $\bar{\Omega} \times [0, \infty)$.
\end{lemma}
Hence, $0\leq S(x,t)\leq M_1$, $0\leq I(x,t)\leq M_2$, $0\leq V(x,t)\leq M_3$ for $(x,t)\in \bar{\Omega} \times [0, \infty)$. And also, by the maximum principle, if  $\varphi_i(x,0)\neq 0, (i=1, 2, 3)$, we have $S(x,t)>0, I(x,t)>0, V(x,t)>0$ for all $t>0, x\in\bar{\Omega}$.

\subsection{Equilibria and Basic reproduction number}
It is easy to verify that system (\ref{e1}) always has a disease-free equilibrium $E_0(S_0, 0, 0)$ with $S_0=\dfrac{\Lambda}{d_S}$, and if exists the endemic equilibrium $E^{*}(S^{*}, I^{*}, V^{*})$ satisfies
\begin{equation}\label{eq1}
\begin{cases}
\begin{split}
    &\Lambda  - S f(V) - Sg(I) - d_S S=0,\\
    &S f(V) +  Sg(I)  - (\gamma+d_I) I=0,\\
    &\alpha I -  d_V  V=0.
\end{split}
\end{cases}
\end{equation}
In order to find the basic reproduction number ($\mathcal{R}_{0}$) for the system (\ref{e1}), we obtain the following linear system at $E_0$ for the infected classes:
\begin{equation}\label{eq2}
\begin{cases}
\begin{split}
&\dfrac{\partial I(x,t)}{\partial t}=\ D_2\Delta I + S_0 f'(0)V +  S_0 g'(0)I  - (\gamma+d_I) I, \quad & x \in \Omega, \ t>0,\\
 &\dfrac{\partial V(x,t)}{\partial t}= \  D_3\Delta V + \alpha I -  d_V  V, \quad &x \in \Omega, \ t>0,\\
&\dfrac{\partial I}{\partial\nu}=\dfrac{\partial V}{\partial\nu}=0, \quad &x \in \partial\Omega, \ t>0.
\end{split}
\end{cases}
\end{equation}
Substituting $I(x,t)=e^{\lambda t}\psi_2(x)$ and $V(x,t)=e^{\lambda t}\psi_3(x)$ into (\ref{eq2}), we obtain the following  cooperative eigenvalue problem:
\begin{equation}\label{eq3}
\begin{cases}
\begin{split}
&\lambda \psi_2(x) =\ D_2\Delta \psi_2(x) + S_0 f'(0) \psi_3(x) +  S_0 g'(0) \psi_2(x)  - (\gamma+d_I) \psi_2(x), \quad & x \in \Omega,\\
 &\lambda \psi_3(x)= \  D_3\Delta  \psi_3(x) + \alpha  \psi_2(x) -  d_V   \psi_3(x), \quad &x \in \Omega,\\
&\dfrac{\partial  \psi_2(x)}{\partial\nu}=\dfrac{\partial  \psi_3(x)}{\partial\nu}=0, \quad &x \in \partial\Omega.
\end{split}
\end{cases}
\end{equation}
By \cite{r31}(Theorem 7.6.1), we conclude that (\ref{eq3}) has a principal eigenvalue $\lambda(S_0, f'(0), g'(0))$ with a positive eigenfunction. Now we are in a position to apply the ideas and the theory in \cite{r32} to define $\mathcal{R}_{0}$ for the model (\ref{e1}). Let $\widetilde{T}:C(\bar{\Omega},\mathbb{R}^{2})\rightarrow C(\bar{\Omega},\mathbb{R}^{2})$ be the solution semigroup of the following reaction-diffusion system:
\begin{equation}\label{eq4}
\begin{cases}
\begin{split}
&\dfrac{\partial I(x,t)}{\partial t}=\ D_2\Delta I - (\gamma+d_I) I, \quad & x \in \Omega, \ t>0,\\
 &\dfrac{\partial V(x,t)}{\partial t}= \  D_3\Delta V + \alpha I -  d_V  V, \quad &x \in \Omega, \ t>0,\\
& I(x,0)=\psi_2(x), V(x,0)=\psi_3(x), \quad & x \in \Omega, \ t>0,\\
&\dfrac{\partial I}{\partial\nu}=\dfrac{\partial V}{\partial\nu}=0, \quad &x \in \partial\Omega.
\end{split}
\end{cases}
\end{equation}
Thus, with initial infection $\Psi(x)=(\psi_2(x), \psi_3(x))$, the distribution of those infection members becomes $\widetilde{T}(t)\Psi(x)$ as time evolves. As in \cite{r32}, the matrices $F$ and $V$ defined as
\begin{align*}
F(x)= \begin{pmatrix}
S_0g'(0) & S_0f'(0)\\
0 & 0
\end{pmatrix}
, \quad
V(x)= \begin{pmatrix}
\gamma+d_I & 0\\
-\alpha & d_V
\end{pmatrix}.
\end{align*}
Therefore, the distribution of total new infections is
\begin{align*}
\int_{0}^{\infty} F(x) \widetilde{T}(t)\Psi(x) dt,
\end{align*}
Then, we define
\begin{align*}
L(\Psi)(x):=\int_{0}^{\infty} F(x) \widetilde{T}(t)\Psi(x) dt= F(x) \int_{0}^{\infty} \widetilde{T}(t)\Psi(x) dt.
\end{align*}
It is clear that $L$ is a positive and continuous operator which maps the initial infection distribution $\Psi$ to the distribution of the total infective members produced during the infection period. Applying the idea of next generation operators \cite{r32}, we define the spectral radius of $L$ as the basic reproduction number
$$\mathcal{R}_{0}:=\rho(L).$$
By some calculations, we obtain that
\begin{align*}
\mathcal{R}_{0}=\dfrac{S_0\alpha f'(0)}{d_V(\gamma+d_I )}+\dfrac{S_0g'(0)}{(\gamma+d_I )}:=\mathcal{R}_{01}+\mathcal{R}_{02},
\end{align*}
where $\mathcal{R}_{01}$ and $\mathcal{R}_{02}$ are partial basic reproduction numbers induced by virus-to-cell transmission and cell-to-cell transmission, respectively. The following theorem now prove regarding the meaningful steady states.

\begin{theorem}
If $\mathcal{R}_{0}<1$, then the disease-free equilibrium $E_{0}(S_0,0,0)$ is the only equilibrium of the system $(\ref{e1})$; when $\mathcal{R}_{0}>1$, it also has a unique endemic equilibrium $E^{*}(S^{*}, I^{*}, V^{*})$ where
\begin{align*}
S^*=\dfrac{\Lambda - (\gamma+d_I)I^*}{d_S} \quad \text{and} \quad V^*=\dfrac{\alpha I^*}{d_V}.
\end{align*}
\end{theorem}
\begin{proof}
It is easy to proof for the case $\mathcal{R}_{0}<1$. Consider $\mathcal{R}_{0}>1$, it follows from (\ref{eq1}) that
\begin{align*}
S=\dfrac{\Lambda - (\gamma+d_I)I}{d_S} \quad \text{and} \quad V=\dfrac{\alpha I}{d_V}.
\end{align*}
Let
\begin{align*}
\mathcal{G}(I)=\dfrac{\Lambda  - (\gamma+d_I)I}{d_S}\bigg(f\bigg(\dfrac{\alpha I}{d_V}\bigg)+g(I)\bigg) - (\gamma+d_I)I,
\end{align*}
with $\mathcal{G}(0)=0$, $\mathcal{G}\Big(\dfrac{\Lambda}{\gamma+d_I}\Big)=-\Lambda<0$ and
\begin{align*}
\mathcal{G}'(0)=\dfrac{\Lambda}{d_S}\bigg(\dfrac{\alpha}{d_V}f'(0)+g'(0)\bigg) - (\gamma+d_I)=(\gamma+d_I)(\mathcal{R}_{0}-1)>0.
\end{align*}
Hence, equation $\mathcal{G}(I)=0$ has at least one positive root $I^* \in \bigg(0, \dfrac{\Lambda}{\gamma+d_I}\bigg)$.  That implies the existence of positive equilibrium of the system (\ref{e1}).
In order to show that the positive equilibrium is unique, we use
\begin{align*}
    &S^* f(V^*) +  S^*g(I^*)=  (\gamma+d_I) I^*,\\
    &\alpha I^* =  d_V  V^*.
\end{align*}
Then
\begin{align*}
\mathcal{G}'(I^*)
=&-\dfrac{\gamma+d_I}{d_S}\bigg(f\bigg(\dfrac{\alpha I^*}{d_V}\bigg)+g(I^*)\bigg) +S^*\bigg(\dfrac{\alpha}{d_V}f'\bigg(\dfrac{\alpha I^*}{d_V}\bigg)+g'(I^*)\bigg) - (\gamma + d_I)\\
=& -\dfrac{\gamma+d_I}{d_S}\bigg(f\bigg(\dfrac{\alpha I^*}{d_V}\bigg)+g(I^*)\bigg) + \dfrac{S^*}{I^*}\bigg(\dfrac{\alpha I^*}{d_V}f'\bigg(\dfrac{\alpha I^*}{d_V}\bigg)-f\bigg(\dfrac{\alpha I^*}{d_V}\bigg)+I^*g'(I^*) - g(I^*)\bigg).
\end{align*}
According to equation (\ref{ee11}), we have
\begin{equation}\label{eq5}
Vf'(V)\leq f(V) \quad \text{and}
\quad
Ig'(I)\leq g(I), \quad \text{for}
\quad
V, \ I \geq 0,
\end{equation}
which implies that $\mathcal{G}'(I^*)<0$. If there exists the second positive equilibrium $E^{\diamond}(S^{\diamond}, I^{\diamond}, V^{\diamond})$, then one has $\mathcal{G}'(I^\diamond)<0$. But which contradict the conditions (\ref{eq5}). This completes the proof.
\end{proof}

\subsection{Local Stability}
Let $ 0=\mu_{0}<\mu_{i}<\mu_{i+1},i=1,2\cdots$ be the eigenvalues of $ -\Delta$ on $\Omega$ with homogeneous Neumann boundary condition, $E(\mu_{i})$ the space of eigenfunctions corresponding to $\mu_{i}$ and $\big\{\phi_{ij}:j=1,2,\cdots,\text {dim}\ E(\mu_{i})\big\}$ an orthogonal basis of $E(\mu_{i})$. Then $\mathbb{X}=[C^{1}(\overline{\Omega})]^{3}$ can be decomposed as
\begin{align*}
 \mathbb{X}=\bigoplus^{\infty}_{i=1}\mathbb{X}_{i},\quad \mathbb{X}_{i}=\bigoplus^{\text {dim}\ E(\mu_{i})}_{i=1}\mathbb{X}_{ij},
\end{align*}
where $\mathbb{X}_{ij}=\big\{\boldsymbol{c}\phi_{ij}:\boldsymbol{c}\in \mathbb{R}^{3}\big\}$. Then we can prove the local stability of equilibrium as in \cite{r33, r34}.
\begin{theorem}
If $\mathcal{R}_{0}<1$, then the disease-free equilibrium $E_{0}$ of system $(\ref{e1})$ is locally asymptotically stable.
\end{theorem}
\begin{proof}
The linearization of system (\ref{e1}) at $E_{0}$ can be expressed by
\begin{align*}
 \dfrac{\partial Z(x,t)}{\partial t}=\mathcal{D}\Delta Z(x,t)+\mathcal{A} Z(x,t),
\end{align*}
where $Z=(S, I, V), \mathcal{D}=\text{diag}(D_{1}, D_{2}, D_{3})$, and
\begin{align*}
 \mathcal{A}=
\begin{pmatrix}
-d_S 		& -S_{0}g'(0) 		&-S_{0}f'(0)\\
0 		& S_{0}g'(0)-(\gamma+d_I) 	&S_{0}f'(0)\\
0 		&\alpha			& -d_V
\end{pmatrix}.
\end{align*}
Therefore, the characteristic equation at $E_{0}$ is
\begin{equation}
(\lambda+ d_S + \mu_{i}D_1)[(\lambda- S_{0} g'(0) + \gamma+ d_I + \mu_{i} D_2) (\lambda + d_V + \mu_{i} D_3) -\alpha S_{0}f'(0)]=0.\label{e2}
\end{equation}
It is obvious that (\ref{e2}) has an eigenvalue $\lambda_{1}=-(d_S+\mu_{i}D_1)$. The other two eigenvalues $\lambda_{2}$ and $\lambda_{2}$ are roots of
\begin{align*}
 \lambda^{2}-[S_{0}g'(0)-(\gamma+d_I+\mu_{i}D_2)-(d_V+\mu_{i}D_3)]\lambda-(S_{0}g'(0)-(\gamma+d_I+\mu_{i}D_2))(d_V+\mu_{i}D_3)-\alpha S_{0}f'(0)=0.
\end{align*}
It is easy to see that
\begin{align*}
 \lambda_{2}+\lambda_{3}
=&\ S_{0}g'(0)-(\gamma+d_I+\mu_{i}D_2)-(d_V+\mu_{i}D_3)\\
=&\  S_{0}g'(0)- (\gamma+d_I+d_V )-\mu_{i} (D_2+ D_3)\\
<&\ S_{0}g'(0)- (\gamma+d_I )-\mu_{i} (D_2+ D_3)\\
=&\ (\gamma+d_I ) (\mathcal{R}_{02}-1) -\mu_{i} (D_2+ D_3),
\end{align*}
and
\begin{align*}
\lambda_{2}\lambda_{3}=
&\ d_V(\gamma+d_I)-d_VS_{0}g'(0)-\alpha S_{0}f'(0)+\mu_{i}D_2(d_V+\mu_{i}D_3)+\mu_{i}D_3(\gamma+d_I-S_{0}g'(0))\\
=&\ d_V(\gamma+d_I)(1-\mathcal{R}_{0})+\mu_{i}D_3 (\gamma+d_I)(1-\mathcal{R}_{02})+\mu_{i}D_2(d_V+\mu_{i}D_3).
\end{align*}
Since $\mathcal{R}_{02}<\mathcal{R}_{0}<1$, we have $\lambda_{2}+\lambda_{3}<1$ and $\lambda_{2}\lambda_{3}>0$. This gives that $\text{Re} (\lambda_{2})<0$ and  $\text{Re} (\lambda_{3})<0$. Thus, all eigenvalues of (\ref{e2}) have a negative real parts when $\mathcal{R}_{0}<1$. Hence, $E_{0}$ is locally asymptotically stable. This completes the proof.
\end{proof}

Now we turn our attention to the endemic equilibrium $E^{*}$.
\begin{theorem}
If $\mathcal{R}_{0}>1$, then the endemic equilibrium $E^{*}$ of system (\ref{e1}) is locally asymptotically stable.
\end{theorem}
\begin{proof}
Linearizing system (\ref{e1}) at $E^{*}$ gives
\begin{align*}
 \dfrac{\partial Z(x,t)}{\partial t}= \mathcal{D}\Delta Z(x,t)+\mathcal{B} Z(x,t),
\end{align*}
where
\begin{align*}
 \mathcal{B}= \begin{pmatrix}
-(f(V^{*})+g(I^{*})+d_S) 	& -S^{*}g'(I^{*}) 				 & -S^{*}f'(V^{*})\\
f(V^{*})+g(I^{*}) 		& S^{*}g'(I^{*})-(\gamma+d_I)	 		& S^{*}f'(V^{*})\\
0 					& \alpha 					& -d_V
\end{pmatrix}.
\end{align*}
Thus, the characteristic equation at $E^{*}$ is
\begin{equation}\label{s1}
Q(\lambda)=\lambda^{3}+Q_{2}\lambda^{2}+Q_{1}\lambda+Q_{0}=0,
\end{equation}
where
\begin{align*}
Q_{2}=& \ f(V^{*})+g(I^{*})+d_S+\mu_{i}D_1-S^{*}g'(I^{*})+\gamma+d_I+\mu_{i}D_2+d_V+\mu_{i}D_{3},\\
Q_{1}=& \ (f(V^{*})+g(I^{*})+d_S+\mu_{i}D_1)(\gamma+d_I+\mu_{i}D_2+d_V+\mu_{i}D_{3})-\alpha S^{*}f'(V^{*})\\
&+\ (\gamma+d_I+\mu_{i}D_2)(d_V+\mu_{i}D_{3})-(d_S+\mu_{i}D_1+d_V+\mu_{i}D_{3})S^{*}g'(I^{*}),\\
Q_{0}=& \ (f(V^{*})+g(I^{*})+d_S+\mu_{i}D_1)(\gamma+d_I+\mu_{i}D_2)(d_V+\mu_{i}D_{3})\\
&-\ (d_S+\mu_{i}D_1)[\alpha S^{*}f'(V^{*})+(d_V+\mu_{i}D_{3})S^{*}g'(I^{*})].
\end{align*}
From (\ref{eq1}) and (\ref{eq5}), we get
\begin{align*}
 \alpha S^{*}f'(V^{*})\leq &\ \alpha S^{*}\dfrac{f(V^{*})}{V^{*}}=(\gamma+d_I)d_V \dfrac{f(V^{*})}{f(V^{*})+g(I^{*})},\\
S^{*}g'(I^{*})\leq &\ S^{*}\dfrac{g(I^{*})}{I^{*}}=(\gamma+d_I) \dfrac{g(I^{*})}{f(V^{*})+g(I^{*})}.
\end{align*}
Hence,
\begin{align*}
 Q_{2}
\geq & \ f(V^{*})+g(I^{*})+d_S+\mu_{i}D_1-(\gamma+d_I) \dfrac{g(I^{*})}{f(V^{*})+g(I^{*})}+\gamma+d_I+\mu_{i}D_2+d_V+\mu_{i}D_{3}\\
=&\ f(V^{*})+g(I^{*})+d_S+\mu_{i}D_1+(\gamma+d_I) \dfrac{f(V^{*})}{f(V^{*})+g(I^{*})}+\mu_{i}D_2+d_V+\mu_{i}D_{3},\\
Q_{1}
\geq & \ (f(V^{*})+g(I^{*})+d_S+\mu_{i}D_1)(\gamma+d_I+\mu_{i}D_2+d_V+\mu_{i}D_{3})-(\gamma+d_I)d_V \dfrac{f(V^{*})}{f(V^{*})+g(I^{*})}\\
&+\ (\gamma+d_I+\mu_{i}D_2)(d_V+\mu_{i}D_{3})-(d_S+\mu_{i}D_1+d_V+\mu_{i}D_{3})(\gamma+d_I) \dfrac{g(I^{*})}{f(V^{*})+g(I^{*})}\\
=& \ \ (f(V^{*})+g(I^{*}))(\gamma+d_I+\mu_{i}D_2+d_V+\mu_{i}D_{3})+(d_S+\mu_{i}D_1)(\mu_{i}D_2+d_V+\mu_{i}D_{3})\\
&+\ \mu_{i}D_2(d_V+\mu_{i}D_{3})+(\gamma+d_I) (d_S+\mu_{i}D_1+\mu_{i}D_{3}) \dfrac{f(V^{*})}{f(V^{*})+g(I^{*})},
\end{align*}
and
\begin{align*}
 Q_{0}
\geq & \ (f(V^{*})+g(I^{*})+d_S+\mu_{i}D_1)(\gamma+d_I+\mu_{i}D_2)(d_V+\mu_{i}D_{3})\\
&-\ (d_S+\mu_{i}D_1)\bigg[(\gamma+d_I)d_V \dfrac{f(V^{*})}{f(V^{*})+g(I^{*})}+(d_V+\mu_{i}D_{3})(\gamma+d_I) \dfrac{g(I^{*})}{f(V^{*})+g(I^{*})}\bigg]\\
=& \ (f(V^{*})+g(I^{*}))(\gamma+d_I+\mu_{i}D_2)(d_V+\mu_{i}D_{3})\\
&+ \ (d_S+\mu_{i}D_1)\bigg[\mu_{i}D_2(d_V+\mu_{i}D_{3})+\mu_{i}D_{3}(\gamma+d_I)\dfrac{f(V^{*})}{f(V^{*})+g(I^{*})}\bigg]>0.
\end{align*}
Hence, $Q_{2}Q_{1}-Q_{0}>0$. Then, by using Routh-Hurwitz criterion we claim that all eigenvalues of (\ref{s1}) have negative real parts. Thus, the endemic equilibrium $E^{*}$ of system (\ref{e1}) is locally asymptotically stable when $\mathcal{R}_{0}>1$. This completes the proof.
\end{proof}

\subsection{Global Stability}
Now, we discuss the global stability of the equilibria for the system (\ref{e1}) by considering Lyapunov functional based on the Volterra function $\Phi(x)=x-1-\ln x$. Clearly, $ \Phi(x)\geq 0$ for all $x>0$ and the equality holds if and only if $x=1$. In presence of diffusion, the aim is to show that every solution of the system (\ref{e1}) with a positive initial value that is different from the equilibrium point will converge to the equilibrium.
\begin{theorem}
 If $\mathcal{R}_{0}\leq 1$, then the disease-free equilibrium $E_{0}$ of system (\ref{e1}) is globally asymptotically stable.
\end{theorem}

\begin{proof}
Define a Lyapunov function
\begin{align*}
\mathcal L(t)= \int\limits_{\Omega} \mathcal L_{1}(x,t) dx,
\end{align*}
where
\begin{align*}
\mathcal L_{1}(x,t) = S_{0}\Phi \bigg (\dfrac{S(x,t)}{S_{0}}\bigg) + I(x,t)+B V(x,t),
\end{align*}
and $B$ is a positive constant to be determined later. Then, along the solutions of the system (\ref{e1}), we have
\begin{align*}
\dfrac{\partial \mathcal L_{1}(x,t)}{\partial t}
=&\ \bigg(1-\dfrac{S_{0}}{S}\bigg)\dfrac{\partial S}{\partial t}+\dfrac{\partial I}{\partial t}+B \dfrac{\partial V}{\partial t}\\
=&\ \bigg(1-\dfrac{S_{0}}{S}\bigg)(D_1\Delta S  + \Lambda  - S f(V) - Sg(I) - d_S S)\\
&\ D_2\Delta I + S f(V) +  Sg(I)  - (\gamma+d_I) I+B D_3\Delta V+B (\alpha I -  d_V  V).
\end{align*}
By equation (\ref{ee11}) and choosing $B = (\gamma+ d_I- S_{0}g'(0))/\alpha$, we obtain
\begin{align*}
\dfrac{\partial \mathcal L_{1}(x,t)}{\partial t}
\leq &  -\dfrac{d_S}{S} (S_{0}-S)^{2}+ \dfrac{d_V (\gamma+ d_I)}{\alpha}(R_{0}-1)V\\
      &+\bigg(1-\dfrac{S_{0}}{S}\bigg)D_1\Delta S+ D_2 \Delta I+ B D_{3}\Delta V.
\end{align*}
Using Green's formula and the Neumann boundary conditions in (\ref{e3}), we obtain
\begin{align*}
\int\limits_{\Omega}\bigg(1-\dfrac{S_{0}}{S}\bigg)D_1\Delta S dx= -D_1\int\limits_{\Omega}\nabla\bigg(1-\dfrac{S_{0}}{S}\bigg) \nabla S dx= -D_1 \int\limits_{\Omega}\dfrac{S_{0}}{S^2}| \nabla S|^2dx \leq 0,
\end{align*}
and
\begin{align*}
\int \limits_{\Omega}\Delta I dx= \int \limits_{\Omega}\Delta V dx=0.
\end{align*}
Using above conditions, we obtain
\begin{align*}
\dfrac{d\mathcal L(t)}{dt}=\int \limits_{\Omega} \dfrac{\partial\mathcal L_{1}(x,t)}{\partial t}dx \leq \int \limits_{\Omega}\bigg (-\dfrac{d_S}{S} (S_{0}-S)^{2}+ \dfrac{d_V (\gamma+ d_I)}{\alpha}(R_{0}-1)V - \dfrac{D_1 S_{0}}{S^2}| \nabla S|^2\bigg)dx.
\end{align*}
Therefore, $\dfrac{d\mathcal L(t)}{dt} \leq 0$ whenever  $\mathcal{R}_{0}\leq 1$. It follows that the largest invariant subset of $\big\{\dfrac{d\mathcal L(t)}{dt}=0\big\}$ is the singleton ${E_{0}}$. By LaSalle's Invariance Principle \cite{r35}, the infection-free equilibrium of the system (\ref{e1}) is globally asymptotically stable when $\mathcal{R}_{0}\leq 1$.
\end{proof}
Next, we turn our attention to show the global stability of the endemic equilibrium $E^{*}$.
\begin{theorem}
Consider a Lyapunov function
\begin{align*}
\mathcal H(t)= \int\limits_{\Omega} \mathcal H_{1}(x,t) dx,
\end{align*}
with
\begin{align*}
\mathcal H_{1}(x,t) = S^{*}\Phi \bigg (\dfrac{S}{S^{*}}\bigg) + I^{*}\Phi \bigg (\dfrac{I}{I^{*}}\bigg)+ \dfrac{S^{*}f(V^{*})}{d_V}\Phi \bigg (\dfrac{V}{V^{*}}\bigg).
\end{align*}
Then, $\mathcal H(t)$ is non-negative and is strictly minimized at the unique equilibrium $(S^{*}, I^{*}, V^{*})$, i.e. it is a valid Lyapunov function. Hence, $E^{*}=(S^{*}, I^{*}, V^{*})$ is globally asymptotically stable.
\end{theorem}
\begin{proof}

According to (\ref{e1}), we have
\begin{align*}
\dfrac{\partial \mathcal H_{1}(x,t)}{\partial t}
=& \ \bigg(1-\dfrac{S^{*}}{S}\bigg)\dfrac{\partial S}{\partial t}+\bigg(1-\dfrac{I^{*}}{I}\bigg)\dfrac{\partial I}{\partial t}+\dfrac{S^{*}f(V^{*})}{d_V V^{*}}\bigg(1-\dfrac{V^{*}}{V}\bigg)\dfrac{\partial V}{\partial t}\\
=&\ \bigg(1-\dfrac{S^{*}}{S}\bigg) (D_1\Delta S  +  d_S S^{*}+  S^{*}f(V^{*}) + S^{*} g(I^{*}) - S f(V) - S g(I) - d_S S )\\
  &+\bigg(1-\dfrac{I^{*}}{I}\bigg)(D_2\Delta I + S f(V) +  Sg(I)  - (\gamma+d_I) I)\\
  &+ \dfrac{S^{*}f(V^{*})}{d_V V^{*}} \bigg(1-\dfrac{V^{*}}{V}\bigg)(D_3\Delta V + \alpha I -  d_V  V)\\
=& \bigg(1-\dfrac{S^{*}}{S}\bigg)(d_S S^{*}-d_S S)+S^{*}g(I^{*})\bigg(2-\dfrac{S^{*}}{S}+\dfrac{g(I)}{g(I^{*})}-\dfrac{S g(I) I^{*}}{S^{*} g(I^{*})I}-\dfrac{I}{I^{*}}\bigg)\\
&+S^{*}f(W^{*})\bigg(3-\dfrac{S^{*}}{S}+\dfrac{f(V)}{f(V^{*})}-\dfrac{Sf(V)I^{*}}{S^{*}f(V^{*})I}-\dfrac{V^{*}I}{VI^{*}}-\dfrac{V}{V^{*}}\bigg)\\
&+\bigg(1-\dfrac{S^{*}}{S}\bigg) D_1 \Delta S +\bigg(1-\dfrac{I^{*}}{I}\bigg)D_2 \Delta I+\dfrac{S^{*}f(V^{*})}{d_V V^{*}}\bigg(1-\dfrac{V^{*}}{V}\bigg)D_{3}\Delta V\\
=&\ -\dfrac{d_S}{S}(S-S^{*})^{2}-S^{*}g(I^{*})\bigg(\Phi\bigg(\dfrac{S^{*}}{S}\bigg)+\Phi\bigg(\dfrac{S g(I)I^{*}}{S^{*}g(I^{*})I}\bigg)+\Phi\bigg(\dfrac{I}{I^{*}}\bigg)-\Phi\bigg(\dfrac{g(I)}{g(I^{*})}\bigg)\bigg)\\
&- S^{*}f(V^{*})\bigg(\Phi\bigg(\dfrac{S^{*}}{S}\bigg)+\Phi\bigg(\dfrac{Sf(V)I^{*}}{S^{*}f(V^{*})I}\bigg)+\Phi\bigg(\dfrac{V^{*}I}{VI^{*}}\bigg)+\Phi\bigg(\dfrac{V}{V^{*}}\bigg)-\Phi\bigg(\dfrac{f(V)}{f(V^{*})}\bigg)\bigg)\\
&+\bigg(1-\dfrac{S^{*}}{S}\bigg) D_1 \Delta S +\bigg(1-\dfrac{I^{*}}{I}\bigg)D_2 \Delta I+\dfrac{S^{*}f(V^{*})}{d_V V^{*}}\bigg(1-\dfrac{V^{*}}{V}\bigg)D_{3}\Delta V.
\end{align*}
Using Green's formula and the Neumann boundary conditions in (\ref{e3}), we obtain
\begin{align*}
\int\limits_{\Omega}\bigg(1-\dfrac{S^{*}}{S}\bigg)D_1\Delta S dx= -D_1\int\limits_{\Omega}\nabla\bigg(1-\dfrac{S^{*}}{S}\bigg) \nabla S dx= -D_1 \int\limits_{\Omega}\dfrac{S^{*}}{S^2}| \nabla S|^2dx \leq 0,
\end{align*}
similarly
\begin{align*}
&\int\limits_{\Omega}\bigg(1-\dfrac{I^{*}}{I}\bigg)D_2\Delta I dx= -D_2\int\limits_{\Omega}\nabla\bigg(1-\dfrac{I^{*}}{I}\bigg) \nabla I dx= -D_2 \int\limits_{\Omega}\dfrac{I^{*}}{I^2}| \nabla I|^2dx \leq 0,\\
&\dfrac{S^{*}f(V^{*})}{d_V V^{*}}\int\limits_{\Omega}\bigg(1-\dfrac{V^{*}}{V}\bigg)D_3\Delta V dx= -\dfrac{D_3 S^{*}f(V^{*})}{d_V V^{*}}\int\limits_{\Omega}\nabla\bigg(1-\dfrac{V^{*}}{V}\bigg) \nabla V dx= -\dfrac{D_3 S^{*}f(V^{*})}{d_V} \int\limits_{\Omega}\dfrac{1}{V^2}| \nabla V|^2dx \leq 0,
\end{align*}
By assumption (A2), we get
\begin{align*}
 \Phi\bigg(\dfrac{f(V)}{f(V^{*})}\bigg)-\Phi\bigg(\dfrac{V}{V^{*}}\bigg)
=&\dfrac{f(V)}{f(V^{*})}-\dfrac{V}{V^{*}}+\ln\bigg(\dfrac{V f(V^{*})}{V^{*}f(V)}\bigg)\\
\leq &\ \dfrac{f(V)}{f(V^{*})}-\dfrac{V}{V^{*}}+ \dfrac{Vf(V^{*})}{V^{*}f(V)}-1\\
=&\ \bigg(\dfrac{f(V)}{f(V^{*})}-\dfrac{V}{V^{*}}\bigg) \bigg(1-\dfrac{f(V^{*})}{f(V)}\bigg)\\
\leq &\ 0.
\end{align*}
Similarly, we have
\begin{align*}
 \Phi\bigg(\dfrac{g(I)}{g(I^{*})}\bigg)-\Phi\bigg(\dfrac{I}{I^{*}}\bigg)\leq \bigg(\dfrac{g(I)}{g(I^{*})}-\dfrac{I}{I^{*}}\bigg) \bigg(1-\dfrac{g(I^{*})}{g(I)}\bigg)\leq 0.
\end{align*}
Using above conditions, we conclude that
\begin{align*}
 \dfrac{d \mathcal H(t)}{dt}=& \int \limits_{\Omega} \dfrac{\partial \mathcal H_{1}(x,t)}{\partial t}dx\\
\leq &\int \limits_{\Omega}\bigg[-\dfrac{d_S}{S}(S-S^{*})^{2}-S^{*}g(I^{*})\bigg(\Phi\bigg(\dfrac{S^{*}}{S}\bigg)
+\Phi\bigg(\dfrac{S g(I)I^{*}}{S^{*}g(I^{*})I}\bigg)\bigg)\\
& -S^{*}f(V^{*})\bigg(\Phi\bigg(\dfrac{S^{*}}{S}\bigg)+\Phi\bigg(\dfrac{Sf(V)I^{*}}{S^{*}f(V^{*})I}\bigg)
+\Phi\bigg(\dfrac{V^{*}I}{VI^{*}}\bigg)\bigg)\\
&- \dfrac{D_1 S^{*}}{S^2}| \nabla S|^2- \dfrac{D_2 I^{*}}{I^2}| \nabla I|^2 - \dfrac{D_3 S^{*}f(V^{*})}{d_V V^2}| \nabla V|^2\bigg]dx\\
\leq & \ 0.
\end{align*}
Furthermore we have $ \dfrac{d \mathcal H(t)}{dt}=0$ only at steady state  $E^{*}=(S^{*}, I^{*}, V^{*})$. Therefore, by Lyapunov's direct method, the steady state solution  $E^{*}$ is globally asymptotically stable.
\end{proof}

\section{Dynamical behavior of the discretized model}\label{sec3}
In preceding section, we have shown that the global asymptotic stability of the equilibria for the continuous system (\ref{e1}) is completely determined by the basic reproduction number $\mathcal{R}_{0}$ by constructing appropriate Lyapunov functionals. This arises a natural question that whether the global asymptotic stability of the equilibria of the discrete system (\ref{a1}) can be preserved. In this section, we will discuss this problem. Clearly, the discretized system (\ref{a1}) has the same two steady states as the continuous system (\ref{e1}). In the following theorem, we show the system (\ref{a1}) is non-negative and bounded.
\begin{theorem} (Positivity and Boundedness)
For any $\Delta t>0$ and $\Delta x>0$, the solution of system (\ref{a1})-(\ref{a3}) is non-negative and bounded for all $k\in \mathbb{N}$.
\end{theorem}
\begin{proof}
The positivity of the solutions of the discretized system (\ref{a1}) can be proved using the M-matrix theory \cite{r26}.  From the first equation of system (\ref{a1}), we get
\begin{align*}
\mathcal A^{k}S^{k+1}=S^{k}+\lambda \Delta t,
\end{align*}
where
\begin{align*}
\mathcal A^{k}=\begin{pmatrix}
a_{0}^{k}			& a 			& 0 			&\cdots 		& 0 & 0 & 0 \\	
a 				& a_{1}^{k} 	& a 			&\cdots 		& 0 & 0 & 0 \\
0 				& a 			& a_{2}^{k} 	&\cdots 		& 0 & 0 & 0 \\
\vdots 			& \vdots 		& \vdots 		& \ddots 		&\vdots &\vdots &\vdots \\
0 				& 0 			& 0 			& \cdots 		& a_{M-2}^{k} & a & 0 \\
0 				& 0 			& 0 			& \cdots 		& a & a_{M-1}^{k} & a \\
0 				& 0 			& 0 			& \cdots 		& 0 & a & a_{M}^{k}
\end{pmatrix},
\end{align*}
$ \lambda=(\Lambda,\Lambda,\cdots,\Lambda)^{T}$ and the coefficients $a=-D_1\Delta t/(\Delta x)^{2},\ a_{0}^{k}=1+ D_1 \Delta t /(\Delta x)^{2}+ \Delta t (f(V_{0}^{k})+g(I_{0}^{k})+ d_S), \ a_{M}^{k}=1+ D_1 \Delta t /(\Delta x)^{2}+ \Delta t (f(V_{M}^{k}) + g(I_{M}^{k})+ d_S)$ and $a_{i}^{k}=1 + 2D_1 \Delta t /(\Delta x)^{2}+ \Delta t (f(V_{i}^{k})+g(I_{i}^{k})+ d_S)$ with $i=1,2,\cdots,M-1$. It is clear that $\mathcal A^{k}$ is a strictly diagonally dominant matrix. Thus, the first equation of system (\ref{a1}) is equivalent to
\begin{align*}
S^{k+1}=(\mathcal A^{k})^{-1}(S^{k}+\lambda \Delta t)>0.
\end{align*}
From the second equation of the system (\ref{a1}), we have
\begin{align*}
\mathcal B I^{k+1}=I^{k}+ \Delta t T^{k+1},
\end{align*}
where $T^{k+1}=(S_{0}^{k+1}(f(V_{0}^{k})+ g(I_{0}^{k})),\ S_{1}^{k+1}(f(V_{1}^{k})+g(I_{1}^{k})),\cdots,
S_{M}^{k+1}(f(V_{M}^{k})+g(I_{M}^{k})))^{T}$ and
\begin{align*}
\mathcal B=\begin{pmatrix}
b_{1}	 	& b_{2}	 & 0 		&\cdots 		& 0 & 0 & 0 \\
b_{2} 		& b_{3} 	& b_{2} 	&\cdots 		& 0 & 0 & 0 \\
0 		& b_{2} 	& b_{3}  	&\cdots 		& 0 & 0 & 0 \\
\vdots 	& \vdots 	& \vdots 	& \ddots 		&\vdots &\vdots &\vdots \\
0 		& 0 		& 0 		& \cdots 		& b_{3} & b_{2} & 0 \\
0 		& 0 		& 0 		& \cdots 		& b_{2} & b_{3} & b_{2} \\
0 		& 0 		& 0 		& \cdots 		& 0 & b_{2} & b_{1}
\end{pmatrix},
\end{align*}
with $b_{1}=1+ D_2 \Delta t /(\Delta x)^{2}+\Delta t (\gamma + d_I),\ b_{2}=- D_2\Delta t/(\Delta x)^{2}$ and $b_{3}=1+ 2D_2 \Delta t /(\Delta x)^{2}+\Delta t (\gamma + d_I)$. Since $\mathcal B$ is a M-matrix, we get
\begin{align*}
I^{k+1}=\mathcal B^{-1}(I^{k}+\Delta t T^{k+1}).
\end{align*}
Similarly, from the third equation of system (\ref{a1}), we have
\begin{align*}
\mathcal C V^{k+1}=V^{k}+ \alpha \Delta t I^{k+1},
\end{align*}
where
\begin{align*}
\mathcal C=\begin{pmatrix}
c_{1} 		& c_{2} 		& 0 			&\cdots 		& 0 & 0 & 0 \\
c_{2} 		& c_{3} 		& c_{2} 		&\cdots 		& 0 & 0 & 0 \\
0 		& c_{2} 		& c_{3}  		&\cdots 		& 0 & 0 & 0 \\
\vdots 	& \vdots 		& \vdots 		& \ddots 		&\vdots &\vdots &\vdots \\
0 		& 0 			& 0 			& \cdots 		& c_{3} & c_{2} & 0 \\
0 		& 0 			& 0 			& \cdots 		& c_{2} & c_{3} & c_{2} \\
0 		& 0 			& 0 			& \cdots 		& 0 & c_{2} & c_{1}
\end{pmatrix}
\end{align*}
with $c_{1}=1+D_{3} \Delta t /(\Delta x)^{2}+d_V \Delta t ,\ c_{2}=-D_{3}\Delta t/(\Delta x)^{2}$ and $c_{3}=1+2D_{3} \Delta t /(\Delta x)^{2}+ d_V \Delta t $. Since $\mathcal C$ is a M-matrix, we obtain
\begin{align*}
V^{k+1}=\mathcal C^{-1}(V^{k}+\alpha \Delta t I^{k+1}).
\end{align*}
Since all parameters in the system (\ref{a1}) are positive, it is easy to see that the solution remains non-negative for all $k\in \mathbb{N}$.\\
Next, we prove the boundedness of the solution. Define a sequence $\{G^{k}\}$ as follows:
\begin{align*}
\mathcal Q^{k}=\sum_{n=0}^{M}(S_{n}^{k}+I_{n}^{k}).
\end{align*}
It follows from the first two equations of system (\ref{a1}) that
\begin{align*}
\mathcal Q^{k+1} - \mathcal Q^{k}&=\Lambda (M+1)\Delta t-d_S \Delta t \sum_{n=0}^{M}S_{n}^{k+1} - (\gamma + d_I)\Delta t \sum_{n=0}^{M}I_{n}^{k+1}\\
& \leq \Lambda (M+1)\Delta t - d \Delta t \mathcal Q^{k+1}.
\end{align*}
where  $d=\text{min}\{d_S, d_I\}$. Hence, we have
\begin{align*}
\mathcal Q^{k+1}\leq \dfrac{\Lambda (M+1)\Delta t}{1+d \Delta t}+\dfrac{\mathcal Q^{k}}{1+d \Delta t}.
\end{align*}
By mathematical induction, we obtain
\begin{align*}
\limsup_{k\rightarrow +\infty} \mathcal Q^{k}\leq N(M+1).
\end{align*}
By the third equation of the system (\ref{a1}), we get
\begin{align*}
\sum_{n=0}^{M}V_{n}^{k+1}=\dfrac{\alpha\Delta t}{1+d_V \Delta t} \sum_{n=0}^{M} I_{n}^{k+1}+\dfrac{1}{1+d_V\Delta t}\sum_{n=0}^{M}V_{n}^{k}.
\end{align*}
Since $\{G^{k}\}$ is bounded, there exists a positive constant $\xi$ such that $\sum_{n=0}^{M}I_{n}^{k}\leq \xi$. Thus, we have
\begin{align*}
\limsup_{k\rightarrow +\infty} \sum_{n=0}^{M}V_{n}^{k+1}\leq \dfrac{\alpha \xi}{d_V}.
\end{align*}
This completes the proof.
\end{proof}
\subsection{Global Stability} In this section, we establish the global asymptotic stability of the steady states $E_0$ and $E^*$ of the discrete system (\ref{a1}), by constructing discrete Lyapunov functions.
\begin{theorem}\label{theorem3.2}
For any $\Delta t>0$ and $\Delta x>0$, if $\mathcal{R}_{0}\leq 1$, then the disease-free equilibrium $E_{0}$ of system (\ref{a1}) is globally asymptotically stable.
\end{theorem}
\begin{proof}
Define a discrete Lyapunov function as follows
\begin{align*}
\mathcal L^{k}=\sum_{n=0}^{M}\dfrac{1}{\Delta t}\left[S_{0}\Phi \left(\dfrac{S_{n}^{k}}{S_{0}}\right)+(1+\rho_{0}\Delta t) I_{n}^{k}+ \rho_{1}(1+\rho_{2}\Delta t)V_{n}^{k}\right],
\end{align*}
where $\rho_{0},\ \rho_{1}$ and $\rho_{2}$ are positive constant to be determined later. Applying the inequality $\ln x \leq x-1$, the difference of $\mathcal L^{k}$ satisfies
\begin{align*}
\mathcal L^{k+1}- \mathcal L^{k} 
= & \sum_{n=0}^{M}\dfrac{1}{\Delta t}\bigg[S_{n}^{k+1}-S_{n}^{k}+S_{0}\ln \left(\dfrac{S_{n}^{k}}{S_{n}^{k+1}}\right)+(1+\rho_{0}\Delta t) (I_{n}^{k+1}-I_{n}^{k})+ \rho_{1}(1+\rho_{2}\Delta t)(V_{n}^{k+1}-V_{n}^{k})\bigg]\\
 \leq& \sum_{n=0}^{M}\dfrac{1}{\Delta t}\bigg[\left(1-\dfrac{S_{0}}{S_{n}^{k+1}}\right)(S_{n}^{k+1}-S_{n}^{k})+(1+\rho_{0}\Delta t) (I_{n}^{k+1}-I_{n}^{k}) + \rho_{1}(1+\rho_{2}\Delta t)(V_{n}^{k+1}-V_{n}^{k})\bigg]\\
 =& \sum_{n=0}^{M}\bigg[\left(1-\dfrac{S_{0}}{S_{n}^{k+1}}\right)(\Lambda - S_{n}^{k+1}f(V_{n}^{k})-S_{n}^{k+1}g(I_{n}^{k})- d_S S_{n}^{k+1})+ S_{n}^{k+1}f(V_{n}^{k})+S_{n}^{k+1}g(I_{n}^{k})\\
&  - (\gamma + d_I) I_{n}^{k+1} + \rho_{0}(I_{n}^{k+1}-I_{n}^{k}) + \rho_{1}(\alpha I_{n}^{k+1}- d_V V_{n}^{k+1})+\rho_{1}\rho_{2}(V_{n}^{k+1}-V_{n}^{k})\bigg]+ \mathcal R^{k}\\
=& \sum_{n=0}^{M}\bigg[-\dfrac{d_S}{S_{n}^{k+1}}(S_{n}^{k+1}-S_{0})^{2}+S_{0}f(V_{n}^{k})+S_{0}g(I_{n}^{k})-(\gamma+d_I)I_{n}^{k+1}+\rho_{0}(I_{n}^{k+1}-I_{n}^{k}) \\
&+ \rho_{1}(\alpha I_{n}^{k+1} - d_V V_{n}^{k+1})+\rho_{1}\rho_{2}(V_{n}^{k+1}-V_{n}^{k})\bigg]+ \mathcal R^{k},
\end{align*}
where
\begin{align*}
\mathcal R^{k}
=& \sum_{n=0}^{M}\dfrac{1}{(\Delta x)^{2}}\bigg[D_1 \left(1-\dfrac{S_{0}}{S_{n}^{k+1}}\right)\left(S_{n+1}^{k+1}-2S_{n}^{k+1}+S_{n-1}^{k+1}\right)+D_2 \left(I_{n+1}^{k+1}-2I_{n}^{k+1}+I_{n-1}^{k+1}\right)\\
&+ \rho_{1}D_{3}\left(V_{n+1}^{k+1}-2V_{n}^{k+1}+V_{n-1}^{k+1}\right)\bigg]\\
=& \sum_{n=0}^{M}\dfrac{1}{(\Delta x)^{2}}\bigg[D_1 \left(S_{n+1}^{k+1}-2S_{n}^{k+1}+S_{n-1}^{k+1}\right)- S_{0}D_1\left(\dfrac{S_{n+1}^{k+1}}{S_{n}^{k+1}}-2+\dfrac{S_{n-1}^{k+1}}{S_{n}^{k+1}}\right)\\
&+D_2 \left(I_{n+1}^{k+1}-2I_{n}^{k+1}+I_{n-1}^{k+1}\right)+ \rho_{1}D_{3} \left(V_{n+1}^{k+1}-2V_{n}^{k+1}+V_{n-1}^{k+1}\right)\bigg].
\end{align*}
Using the arithmetic-geometric inequality
\begin{equation}\label{a4}
2-\dfrac{S_{n+1}^{k+1}}{S_{n}^{k+1}}-\dfrac{S_{n}^{k+1}}{S_{n+1}^{k+1}}\leq 0,\quad \text{for}\ n\in \{0,1,\cdots,M-1\},
\end{equation}
we get
\begin{align*}
\mathcal R^{k}
\leq & \dfrac{1}{(\Delta x)^{2}}\bigg[D_1\left(S_{M+1}^{k+1}-S_{M}^{k+1}+S_{-1}^{k+1}-S_{0}^{k+1}\right)
-S_{0}D_1 \left(\dfrac{S_{-1}^{k+1}}{S_{0}^{k+1}}-2+\dfrac{S_{M+1}^{k+1}}{S_{M}^{k+1}}\right)\\
&+D_2 \left(I_{M+1}^{k+1}-I_{M}^{k+1}+I_{-1}^{k+1}-I_{0}^{k+1}\right)+ \rho_{1}D_{3}\left(V_{M+1}^{k+1}-V_{M}^{k+1}+V_{-1}^{k+1}-W_{0}^{k+1}\right)\bigg]\\
=&\ 0.
\end{align*}
By using equation (\ref{ee11}), we obtain
\begin{align*}
\mathcal L^{k+1}- \mathcal L^{k}\leq &\sum_{n=0}^{M}\bigg[-\dfrac{d_S}{S_{n}^{k+1}}(S_{n}^{k+1}-S_{0})^{2}+S_{0}f'(0)V_{n}^{k}+S_{0}g'(0)I_{n}^{k}-(\gamma+d_I)I_{n}^{k+1}+\rho_{0}(I_{n}^{k+1}-I_{n}^{k})\\
&+ \rho_{1}(\alpha I_{n}^{k+1}-d_V V_{n}^{k+1})+\rho_{1}\rho_{2}(V_{n}^{k+1}-V_{n}^{k})\bigg].
\end{align*}
Letting $\rho_{0}=S_{0}g'(0),\ \rho_{1}=(\gamma + d_I - S_{0}g'(0))/\alpha$ and $\rho_{2}=d_V$, we get
\begin{align*}
\mathcal L^{k+1}- \mathcal L^{k}\leq &\sum_{n=0}^{M}\bigg[-\dfrac{d_S}{S_{n}^{k+1}}(S_{n}^{k+1}-S_{0})^{2}+\dfrac{d_V(\gamma+d_I)}{\alpha}(\mathcal{R}_{0}-1)V_{n}^{k}\bigg].
\end{align*}
If $\mathcal{R}_{0}\leq 1$, for all $k \in\mathbb{N}$ we get
\begin{align*}
\mathcal L^{k+1}- \mathcal L^{k}\leq 0,
\end{align*}
which gives that $\{\mathcal L^{k}\}_{k \in\mathbb{N}}$ is a monotone decreasing sequence. That means there exists a constant $\tilde{\mathcal L}\geq 0$ such that $\lim_{k\rightarrow +\infty} \mathcal L^{k}=\tilde{\mathcal L}$ and we have $\lim_{k\rightarrow +\infty} (\mathcal L^{k+1}- \mathcal L^{k})=0$. Arguing on the lines of \cite{r26} to system (\ref{a1}), we obtain
\begin{align*}
\lim_{k\rightarrow +\infty} S_{n}^{k}=S_{0}, \ \lim_{k\rightarrow +\infty} I_{n}^{k}=0 \quad \text{and} \quad \lim_{k\rightarrow +\infty} V_{n}^{k}=0,
\end{align*}
for all $n\in \{0,1,\cdots,M\}$ and when $\mathcal{R}_{0}\leq 1$. This completes the proof.
\end{proof}
Next, we discuss the global stability of the endemic equilibrium $E^{*}$ when $\mathcal{R}_{0}>1$.
\begin{theorem}\label{theorem3.3}
For any $\Delta t>0$ and $\Delta x>0$, then the endemic equilibrium $E^{*}$ of system (\ref{a1}) is globally asymptotically stable when $\mathcal{R}_{0}> 1$,.
\end{theorem}
\begin{proof}
We define the following discretized Lyapunov function
\begin{align*}
\mathcal H^{k}=\sum_{n=0}^{M}\dfrac{1}{\Delta t}\bigg[S^{*}\Phi \left(\dfrac{S_{n}^{k}}{S^{*}}\right)+(I^{*}+S^{*}g(I^{*})\Delta t)\Phi \left(\dfrac{I_{n}^{k}}{I^{*}}\right)+\dfrac{S^{*}f(V^{*})}{d_V}(1+d_V \Delta t)\Phi\left(\dfrac{V_{n}^{k}}{V^{*}}\right) \bigg].
\end{align*}
Obviously, $\mathcal H^{k}\geq 0$ for all $k\in \mathbb{N}$ with the equality holds if and only if $S_{n}^{k}=S^{*},\ I_{n}^{k}=I^{*}$ and $V_{n}^{k}=V^{*}$ for all $n\in \{0,1,\cdots,M\}$ and $k\in \mathbb{N}$. Using $\ln x \leq x-1$, the difference of $\mathcal H^{k}$ satisfies
\begin{align*}
\mathcal H^{k+1}& - \mathcal H^{k}\\ 
=& \sum_{n=0}^{M}\dfrac{1}{\Delta t}\bigg[S_{n}^{k+1}-S_{n}^{k}+S^{*}\ln \left(\dfrac{S_{n}^{k}}{S_{n}^{k+1}}\right)+I_{n}^{k+1}-I_{n}^{k}-I^{*}\ln \left(\dfrac{I_{n}^{k}}{I_{n}^{k+1}}\right)\\
&+ S^{*}g(I^{*})\Delta t\left(\dfrac{I_{n}^{k+1}}{I^{*}}-\dfrac{I_{n}^{k}}{I^{*}}+\ln \left(\dfrac{I_{n}^{k}}{I_{n}^{k+1}}\right)\right) + \dfrac{S^{*}f(V^{*})}{d_V}\left(\dfrac{V_{n}^{k+1}}{V^{*}}-\dfrac{V_{n}^{k}}{V^{*}}+\ln \left(\dfrac{V_{n}^{k}}{V_{n}^{k+1}}\right)\right)\\
&+ S^{*}f(V^{*})\Delta t\left(\dfrac{V_{n}^{k+1}}{V^{*}}-\dfrac{V_{n}^{k}}{V^{*}}+\ln \left(\dfrac{V_{n}^{k}}{V_{n}^{k+1}}\right)\right)\bigg]\\
\leq & \sum_{n=0}^{M}\dfrac{1}{\Delta t}\bigg[\left(1-\dfrac{S^{*}}{S_{n}^{k+1}}\right)(S_{n}^{k+1}-S_{n}^{k})+
\left(1-\dfrac{I^{*}}{I_{n}^{k+1}}\right)(I_{n}^{k+1}-I_{n}^{k})\\
& +S^{*}g(I^{*})\Delta t\left(\dfrac{I_{n}^{k+1}}{I^{*}}-\dfrac{I_{n}^{k}}{I^{*}}+\ln \left(\dfrac{I_{n}^{k}}{I_{n}^{k+1}}\right)\right) +\dfrac{S^{*}f(V^{*})}{d_V V^{*}}\left(1-\dfrac{V^{*}}{V_{n}^{k+1}}\right)(V_{n}^{k+1}+V_{n}^{k})\\
& + S^{*}f(V^{*})\Delta t \left(\dfrac{V_{n}^{k+1}}{V^{*}}-\dfrac{V_{n}^{k}}{V^{*}}+\ln \left(\dfrac{V_{n}^{k}}{V_{n}^{k+1}}\right)\right)\bigg]\\
=& \sum_{n=0}^{M}\bigg[\left(1-\dfrac{S^{*}}{S_{n}^{k+1}}\right)(\Lambda - S_{n}^{k+1}f(V_{n}^{k})- S_{n}^{k+1}g(I_{n}^{k})- d_S S_{n}^{k+1})\\
&+ \left(1-\dfrac{I^{*}}{I_{n}^{k+1}}\right)(S_{n}^{k+1}f(V_{n}^{k})+ S_{n}^{k+1}g(I_{n}^{k})- (\gamma + d_I)I_{n}^{k+1})+ \dfrac{S^{*}f(V^{*})}{d_V V^{*}}\left(1-\dfrac{V^{*}}{V_{n}^{k+1}}\right)(\alpha I_{n}^{k+1}- d_V V_{n}^{k+1})\\
&+S^{*}g(I^{*})\left(\dfrac{I_{n}^{k+1}}{I^{*}}-\dfrac{I_{n}^{k}}{I^{*}}+\ln \left(\dfrac{I_{n}^{k}}{I_{n}^{k+1}}\right)\right)+ S^{*}f(V^{*})\left(\dfrac{V_{n}^{k+1}}{V^{*}}-\dfrac{V_{n}^{k}}{V^{*}}+\ln \left(\dfrac{V_{n}^{k}}{V_{n}^{k+1}}\right)\right)\bigg]+\mathcal Z^{k},
\end{align*}
where
\begin{align*}
\mathcal Z^{k}=& \sum_{n=0}^{M}\dfrac{1}{(\Delta x)^{2}}\bigg[D_1 \left(1-\dfrac{S^{*}}{S_{n}^{k+1}}\right)\left(S_{n+1}^{k+1}-2S_{n}^{k+1}+S_{n-1}^{k+1}\right)+ D_2 \left(1-\dfrac{I^{*}}{I_{n}^{k+1}}\right)\left(I_{n+1}^{k+1}-2I_{n}^{k+1}+I_{n-1}^{k+1}\right)\\
&+ \dfrac{S^{*}f(V^{*})D_3}{d_V V^{*}}\left(1-\dfrac{V^{*}}{V_{n}^{k+1}}\right)\left(V_{n+1}^{k+1}- 2V_{n}^{k+1}+ V_{n-1}^{k+1}\right)\bigg]\\
=& \sum_{n=0}^{M}\dfrac{1}{(\Delta x)^{2}}\bigg[D_1 \left(S_{n+1}^{k+1}-2S_{n}^{k+1}+S_{n-1}^{k+1}\right) - S^{*}D_1 \left(\dfrac{S_{n+1}^{k+1}}{S_{n}^{k+1}} - 2+\dfrac{S_{n-1}^{k+1}}{S_{n}^{k+1}}\right)\\
& + D_2 \left(I_{n+1}^{k+1} - 2I_{n}^{k+1}+I_{n-1}^{k+1}\right) - I^{*}D_2 \left(\dfrac{I_{n+1}^{k+1}}{I_{n}^{k+1}}-2+\dfrac{I_{n-1}^{k+1}}{I_{n}^{k+1}}\right)\\
&+ \dfrac{S^{*}f(V^{*})D_3}{d_V V^{*}}\left(V_{n+1}^{k+1}-2V_{n}^{k+1}+V_{n-1}^{k+1}\right) -\dfrac{S^{*}f(V^{*})D_3}{d_V}\left(\dfrac{V_{n+1}^{k+1}}{V_{n}^{k+1}}-2+\dfrac{V_{n-1}^{k+1}}{V_{n}^{k+1}}\right)\bigg].
\end{align*}
Using the similar arithmetic-geometric inequality (\ref{a4}) for $S, \ I$ and $V$, we get 
\begin{align*}
\mathcal Z^{k}\leq & \dfrac{1}{(\Delta x)^{2}}\bigg[D_1 \left(S_{M+1}^{k+1}-S_{M}^{k+1}+S_{-1}^{k+1}-S_{0}^{k+1}\right) - S^{*}D_1\left(\dfrac{S_{-1}^{k+1}}{S_{0}^{k+1}}-2+\dfrac{S_{M+1}^{k+1}}{S_{M}^{k+1}}\right)\\
&+D_2 \left(I_{M+1}^{k+1}-I_{M}^{k+1}+I_{-1}^{k+1}-I_{0}^{k+1}\right) - I^{*}D_2\left(\dfrac{I_{-1}^{k+1}}{I_{0}^{k+1}}
-2+\dfrac{I_{M+1}^{k+1}}{I_{M}^{k+1}}\right)\\
&+ \dfrac{S^{*}f(V^{*})D_3}{d_V V^{*}}\left(V_{M+1}^{k+1}-V_{M}^{k+1}+V_{-1}^{k+1}-V_{0}^{k+1}\right)-\dfrac{S^{*}f(V^{*})D_3}{d_V}\left(\dfrac{V_{-1}^{k+1}}{V_{0}^{k+1}}-2+\dfrac{V_{M+1}^{k+1}}{V_{M}^{k+1}}\right)\bigg]\\
=& \ 0.
\end{align*}
By assumption (A2), we have the following inequality
\begin{align*}
\Phi \left(\dfrac{f(V_{n}^{k})}{f(V^{*})}\right)-\Phi \left(\dfrac{V_{n}^{k}}{V^{*}}\right)=& \dfrac{f(V_{n}^{k})}{f(V^{*})}-\dfrac{V_{n}^{k}}{V^{*}}+\ln \left(\dfrac{V_{n}^{k}f(V^{*})}{V^{*}f(V_{n}^{k})}\right)\\
\leq & \dfrac{f(V_{n}^{k})}{f(V^{*})}-\dfrac{V_{n}^{k}}{V^{*}}+ \dfrac{V_{n}^{k}f(V^{*})}{V^{*}f(V_{n}^{k})}-1)\\
=& \left(\dfrac{f(V_{n}^{k})}{f(V^{*})}-\dfrac{V_{n}^{k}}{V^{*}}\right)\left(1- \dfrac{f(V^{*})}{f(V_{n}^{k})}\right)\\
\leq & \ 0.
\end{align*}
Similarly, we get
\begin{align*}
\Phi \left(\dfrac{g(I_{n}^{k})}{g(I^{*})}\right)-\Phi \left(\dfrac{I_{n}^{k}}{I^{*}}\right)\leq \left(\dfrac{g(I_{n}^{k})}{g(I^{*})}-\dfrac{I_{n}^{k}}{I^{*}}\right)\left(1- \dfrac{g(I^{*})}{g(I_{n}^{k})}\right)\leq \ 0.
\end{align*}
Hence, using above conditions, we conclude that 
\begin{align*}
\mathcal H^{k+1}& - \mathcal H^{k}\\
=&  \sum_{n=0}^{M}\bigg[-\dfrac{d_S}{S_{n}^{k+1}}(S_{n}^{k+1}-S^{*})^{2} - S^{*}g(I^{*})\bigg(\Phi\left(\dfrac{S^{*}}{S_{n}^{k+1}}\right)
+\Phi\left(\dfrac{S_{n}^{k+1}g(I_{n}^{k})I^{*}}{S^{*}g(I^{*})I_{n}^{k+1}}\right)+ \Phi\left(\dfrac{I_{n}^{k}}{I^{*}}\right)-\Phi \left(\dfrac{g(I_{n}^{k})}{g(I^{*})}\right)\bigg)\\
& -S^{*}f(V^{*})\bigg(\Phi\left(\dfrac{S^{*}}{S_{n}^{k+1}}\right)+\Phi\left(\dfrac{S_{n}^{k+1}f(V_{n}^{k})I^{*}}{S^{*}f(V^{*})I_{n}^{k+1}}\right)+ \Phi\left(\dfrac{V^{*}I_{n}^{k+1}}{V_{n}^{k+1}I^{*}}\right)+ \Phi\left(\dfrac{V_{n}^{k}}{V^{*}}\right)-\Phi \left(\dfrac{f(V_{n}^{k})}{f(V^{*})}\right)\bigg)\bigg]+ \mathcal Z^{k}\\
\leq &\sum_{n=0}^{M}\bigg[-\dfrac{d_S}{S_{n}^{k+1}}(S_{n}^{k+1}-S^{*})^{2} - S^{*}g(I^{*})\bigg(\Phi\left(\dfrac{S^{*}}{S_{n}^{k+1}}\right)
+\Phi\left(\dfrac{S_{n}^{k+1}g(I_{n}^{k})I^{*}}{S^{*}g(I^{*})I_{n}^{k+1}}\right)\bigg)\\
& -S^{*}f(V^{*})\bigg(\Phi\left(\dfrac{S^{*}}{S_{n}^{k+1}}\right) +\Phi\left(\dfrac{S_{n}^{k+1}f(V_{n}^{k})I^{*}}{S^{*}f(V^{*})I_{n}^{k+1}}\right)
+ \Phi\left(\dfrac{V^{*}I_{n}^{k+1}}{V_{n}^{k+1}I^{*}}\right)\bigg)\bigg]\\
\leq & \ 0.
\end{align*}
This implies that $\mathcal H^{k}$ is a monotone decreasing sequence, then there exists a constant $\tilde{\mathcal H}$ such that $\lim_{k\rightarrow \infty}\mathcal H^{k}=\tilde{\mathcal H}$, and we have $\lim_{k\rightarrow \infty}(\mathcal H_{k+1}- \mathcal H^{k})=0.$
Which conclude that
\begin{align*}
\lim_{k\rightarrow \infty}S_{n}^{k}=S^{*},\ \lim_{k\rightarrow \infty}\dfrac{g(I_{n}^{k})I^{*}}{g(I^{*})I_{n}^{k+1}}=
\lim_{k\rightarrow \infty}\dfrac{f(V_{n}^{k})I^{*}}{f(V^{*})I_{n}^{k+1}}=\lim_{k\rightarrow \infty}\dfrac{V^{*}I_{n}^{k+1}}{V_{n}^{k+1}I^{*}}=1.
\end{align*}
By the first equation of system (\ref{a1}), we obtain
\begin{align*}
\dfrac{S_{n}^{k+1} - S_{n}^{k}}{\Delta t}=& D_1 \dfrac{S_{n+1}^{k+1} - 2S_{n}^{k+1} + S_{n-1}^{k+1}}{(\Delta x)^{2}}+\Lambda -S_{n}^{k+1}f(V_{n}^{k}) - S_{n}^{k+1}g(I_{n}^{k}) - d_S S_{n}^{k+1}\\
=& D_1 \dfrac{S_{n+1}^{k+1}-2S_{n}^{k+1}+S_{n-1}^{k+1}}{(\Delta x)^{2}} + \Lambda - d_S S_{n}^{k+1} - I_{n}^{k+1}\left(\dfrac{S_{n}^{k+1}f(V_{n}^{k})}{I_{n}^{k+1}}+\dfrac{S_{n}^{k+1}g(I_{n}^{k})}{I_{n}^{k+1}}\right).
\end{align*}
Taking $k\rightarrow +\infty$ in the above equality, we have
\begin{align*}
0=0+ \Lambda - d_S S^{*}-\lim_{k\rightarrow +\infty}I_{n}^{k+1}\left(\dfrac{S^{*}f(V^{*})}{I^{*}}+\dfrac{S^{*}g(I^{*})}{I^{*}}\right).
\end{align*}
Using $\Lambda - d_S S^{*}= S^{*}f(V^{*})+S^{*}g(I^{*})$, we have
\begin{align*}
\lim_{k\rightarrow +\infty}I_{n}^{k+1}=I^{*}.
\end{align*}
Similarly, we get
\begin{align*}
\lim_{k\rightarrow +\infty}V_{n}^{k}=V^{*}.
\end{align*}
This completes the proof.
\end{proof}
Thus, using the Theorems \ref{theorem3.2} and \ref{theorem3.3}, we conclude that the discretized system (\ref{a1}) exhibits dynamic consistency with the continuous system (\ref{e1}) as well as the global asymptotic stability of both the uninfected and the infected steady states.

\section{Numerical Results}\label{sec4}
In this section, we present numerical example that illustrate and confirm the findings of this study for the linear incidence function such as $f(V)=\beta_1 V$ and $g(I)=\beta_2 I$. Thus the system (\ref{e1}) becomes
\begin{equation}\label{b1}
\begin{cases}
\begin{split}
    &\dfrac{\partial S(x,t)}{\partial t}=\ D_1\Delta S  + \Lambda  - \beta_1 S V - \beta_2 S I - d_S S, \quad x \in \Omega, \ t>0,\\
&\dfrac{\partial I(x,t)}{\partial t}=\ D_2\Delta I + \beta_1 S V +  \beta_2 S I - (\gamma+d_I) I, \quad x \in \Omega, \ t>0,\\
 & \dfrac{\partial V(x,t)}{\partial t}= \  D_3\Delta V + \alpha I -  d_V  V, \quad x \in \Omega, \ t>0,\\
&S(x,0)=\varphi_1(x)\geq 0,\quad I(x,0) = \varphi_2(x),\quad V(x,0)=\varphi_3(x), \quad x \in \Omega,\\
&\dfrac{\partial S}{\partial\nu}=\dfrac{\partial I}{\partial\nu}=\dfrac{\partial V}{\partial\nu}=0,\quad x \in \partial\Omega, \ t>0,
\end{split}
\end{cases}
\end{equation}
The basic reproduction number of the system  (\ref{b1}) is
\begin{align*}
\mathcal{R}_{0}=\dfrac{\Lambda \alpha \beta_1}{d_S d_V(\gamma+d_I )}+\dfrac{\Lambda \beta_2}{d_S(\gamma+d_I )}.
\end{align*}
When $\mathcal{R}_{0}<1$, the system (\ref{b1}) possesses an uninfected steady state $E_0\left(\dfrac{\Lambda}{d_S}, 0, 0\right)$ and also has an infected steady state $E^*\Bigg(\dfrac{\Lambda}{d_S \mathcal{R}_{0}},\dfrac{\Lambda\Big(1-\dfrac{1}{\mathcal{R}_{0}}\Big)}{\gamma +d_I},\dfrac{\alpha \Lambda\Big(1-\dfrac{1}{\mathcal{R}_{0}}\Big)}{d_V(\mu + d_I)}\Bigg)$ for $\mathcal{R}_{0}>1$. At first sensitivity analysis is used to determine the response of the model to variations in its parameter values. In the present case, focus is given to determining how changes in the model parameters impact the basic reproduction number. This is done through the Latin hypercube sampling (LHS) and the partial correlation coefficients (PRCC) to determine the relative importance of the parameters in $\mathcal{R}_{0}$ for the disease transmission \cite{r36}. In such a scenario, it is more appropriate to treat each parameter as a random variable, distributed according to an appropriate probability distribution. We assume that our model parameters are normally distributed although it is quite possible that some parameters are constant towards a particular value such as recruitment rate ($\Lambda$) and death rate ($d_T$) of susceptible cells. PRCC reduces the non-linearity effects by rearranging the data in ascending order, replacing the values with their ranks and then providing the measure of monotonicity after the removal of the linear effects of each model parameter keeping all other parameters constant \cite{r37}. The corresponding Tornado plots based on a random sample of 1000 points for the six parameters in $\mathcal{R}_{0}$ are shown in Figure \ref{figsen}. The horizontal lines represent the significant range of correlation, i.e., $|$PRCC$|> 0.5$. The sensitivity analysis suggests that the most significant parameters are $\beta_1$ and $\beta_2$, an increase in these values will have an increase in the spread of the disease. Hence, these parameters should be estimated with precision to accurately capture the dynamics of the infection.
\begin{figure}[h!]
\centering
\includegraphics[width=5in]{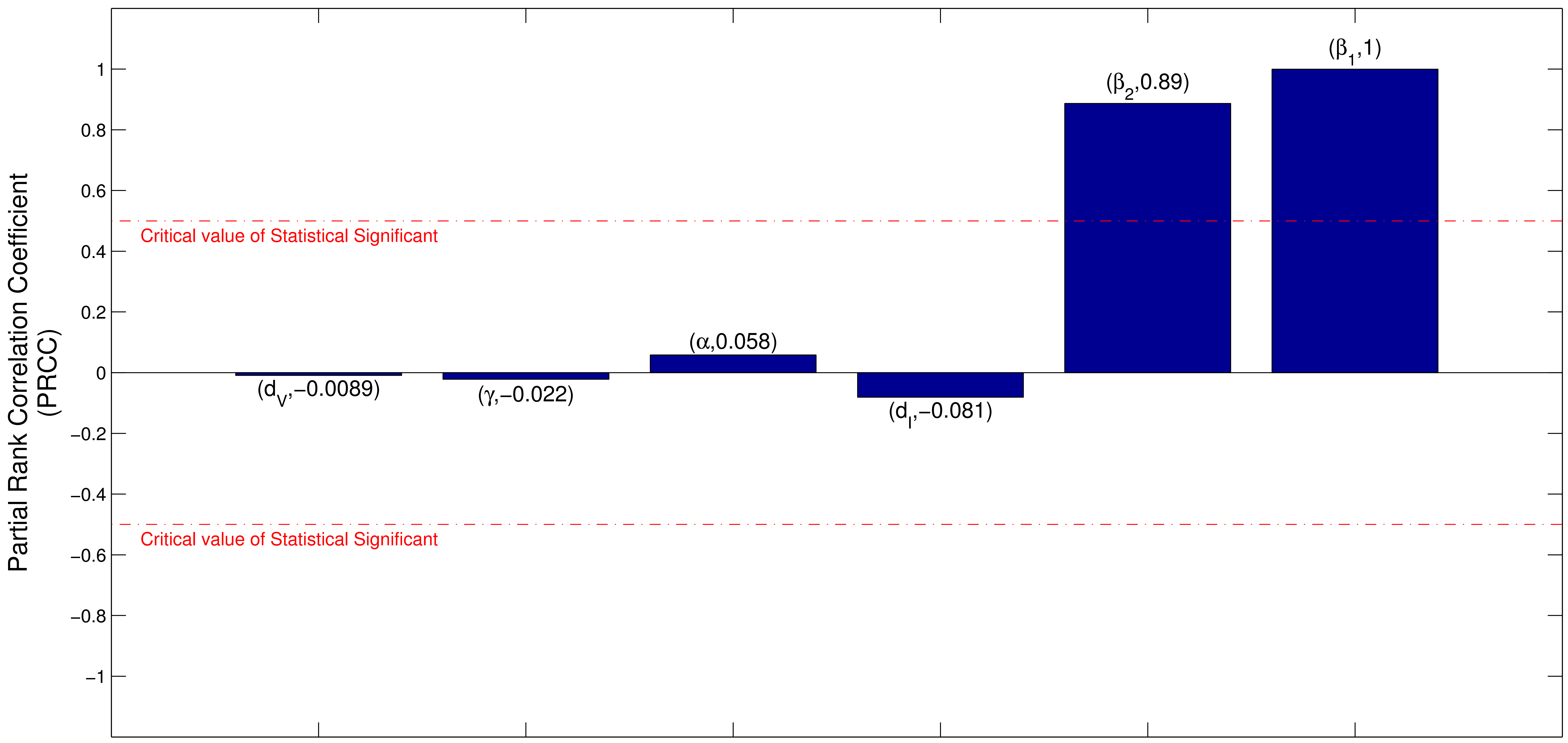}
\caption{Partial rank correlation coefficient (PRCC) results for significance of parameters involved in $\mathcal{R}_{0}$.}
 \label{figsen}
\end{figure}

Now, we numerically illustrate the results for global stability for both the steady states. Accordingly, we use two sets of system parameters, one corresponding to $\mathcal{R}_{0}<1$ (when $E_0$ is globally asymptotically stable for both the continuous and discretized models) and the other for $\mathcal{R}_{0}>1$ (when $E^*$ is globally asymptotically stable for both the continuous and discretized models). The numerical simulation is carried out using the NSFD scheme described by the system (\ref{a1}) with initial condition taken as
$$S(x,0)=10^7,\ \ I(x,0)=100 e^x,\ \ V(x,0)=100 e^x.$$
For the purpose of illustration of both the scenarios, we choose the equal diffusion coefficients as $D_1=D_2=1\ \text{mm}^2 \text{d}^{-1}$  and $D_3=1 \ \text{mm}^2 \text{d}^{-1}$ \cite{r38}. The one-dimensional spatial domain is taken as $\Omega=[0,50]$ and the simulation carried out for a time window of $100$ days. The grid sizes used in the spatial and temporal directions are $\Delta x=0.5$ and $\Delta t=1$, respectively. The parameter set $\Lambda=10^7 \ \text{cells}\ \text{d}^{-1},\  \beta_1=\beta_2=5\times 10^{-12}\ \text{virion}^{-1}\ \text{d}^{-1},\   d_S=0.1\ \text{d}^{-1},\  \gamma=0.01 \ \text{d}^{-1},\  d_I=0.04\ \text{d}^{-1},\  \alpha=100\ \text{d}^{-1},\  d_V=5\ \text{d}^{-1}$ \cite{r26}, results in $\mathcal{R}_{0}=0.21<1.$  Thus, in this case, the uninfected steady state $E^0$ is globally asymptotically stable. It can be observed from Figure \ref{fig1} that this is indeed the case and the system eventually approaches the uninfected steady state $E_0=(10^8, \ 0, \ 0)$. 
\begin{figure}[h!]
\centering
\includegraphics[width=\linewidth]{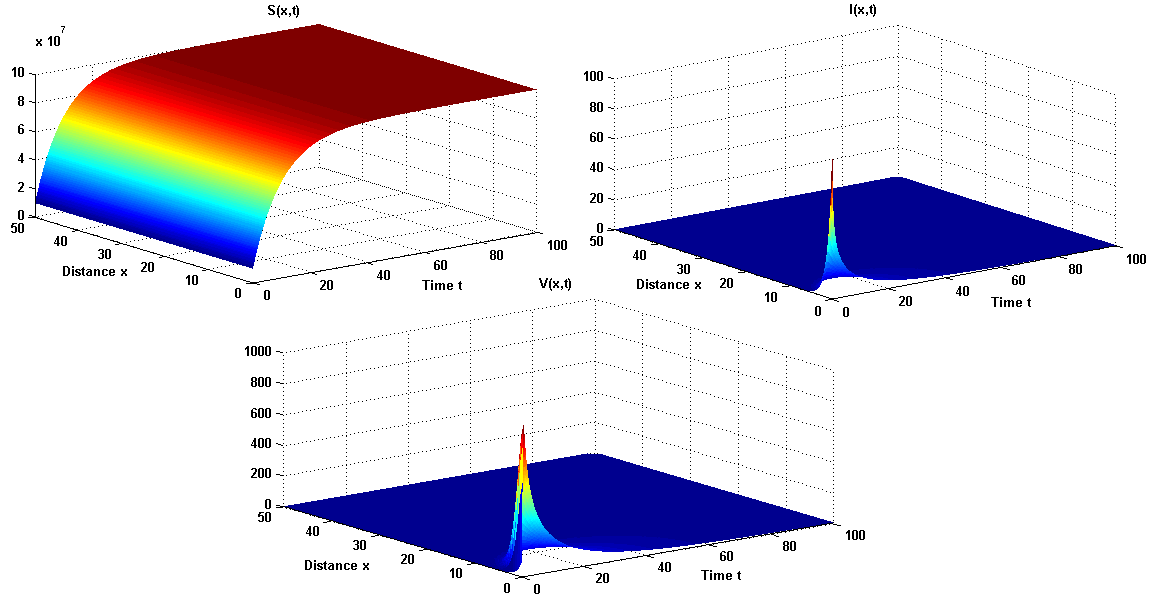}
\caption{When $\mathcal{R}_{0}=0.21<1$, the disease-free equilibrium $E_0$ of system (\ref{b1}) is globally asymptotically stable.}
 \label{fig1}
\end{figure}

For the other scenario, all the parameter values are identical with the exception of $\beta_1=\beta_2=3\times 10^{-10}\ \text{virion}^{-1}\ \text{d}^{-1}$ \cite{r26} which renders $\mathcal{R}_{0}=12.59>1$. In this case, the infected steady state is stable as can be observed numerically in Figure \ref{fig2}, where the state variables approach the infected steady state $E^*=(8 \times10^6, \ 1.8 \times10^8, \ 3.7 \times10^9)$. For both the sets of simulations it can be easily seen that the steady states do not depend on the initial spatial points. 
\begin{figure}[h!]
\centering
\includegraphics[width=\linewidth]{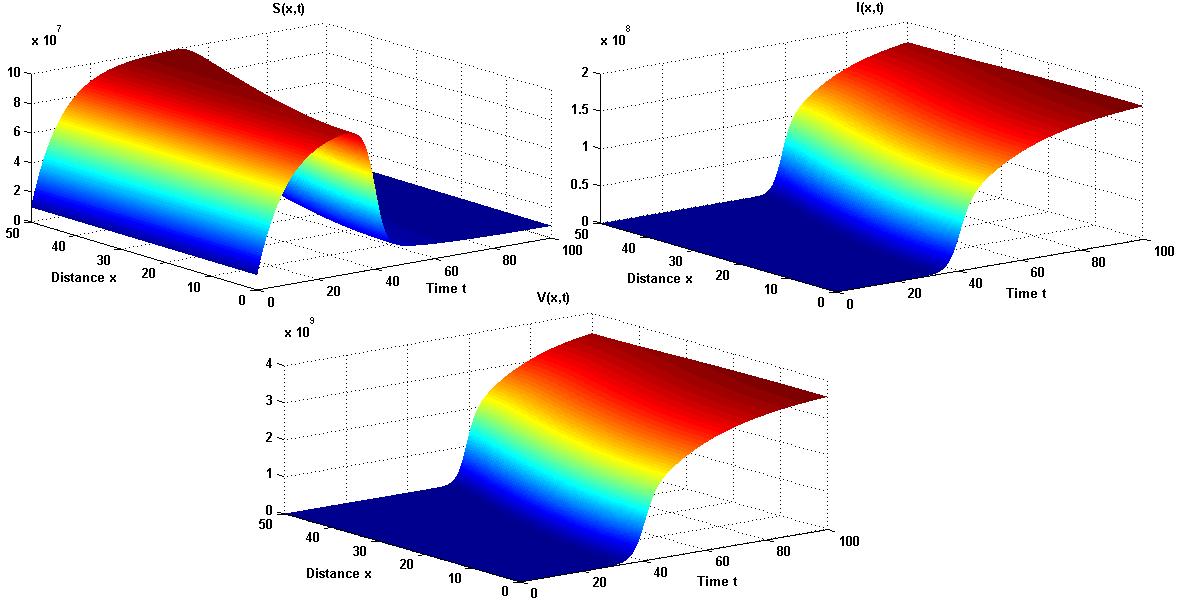}
\caption{When $\mathcal{R}_{0}= 12.59>1$,the disease-free equilibrium $E^*$ of system (\ref{b1}) is globally asymptotically stable.}
 \label{fig2}
\end{figure}

We have also shown that the global asymptotically stable results are dependent only on the parameters of the non-diffusive system because of $\mathcal{R}_{0}$  and independent of the choices of the diffusion coefficients. This is also illustrated by way of numerical simulations. For illustrative purpose we only show the case for $\mathcal{R}_{0}>1$  using the corresponding parameter values used above. We extend our diffusion coefficient  to $D_1=D_2=D_3=100\ \text{mm}^2 \text{d}^{-1}$ and see from Figure \ref{fig3}, that the steady states of the model dynamics are very similar to each other in the long run. Similar results can be observed for different combinations of $(D_1, \ D_2, \ D_3)$. Another crucial advantage of using the NSFD scheme over standard finite difference (SFD) scheme is that the positivity of solutions for long time simulation which already been demonstrated in many works \cite{r26, r39}.

\begin{figure}[h!]
\centering
\includegraphics[width=\textwidth]{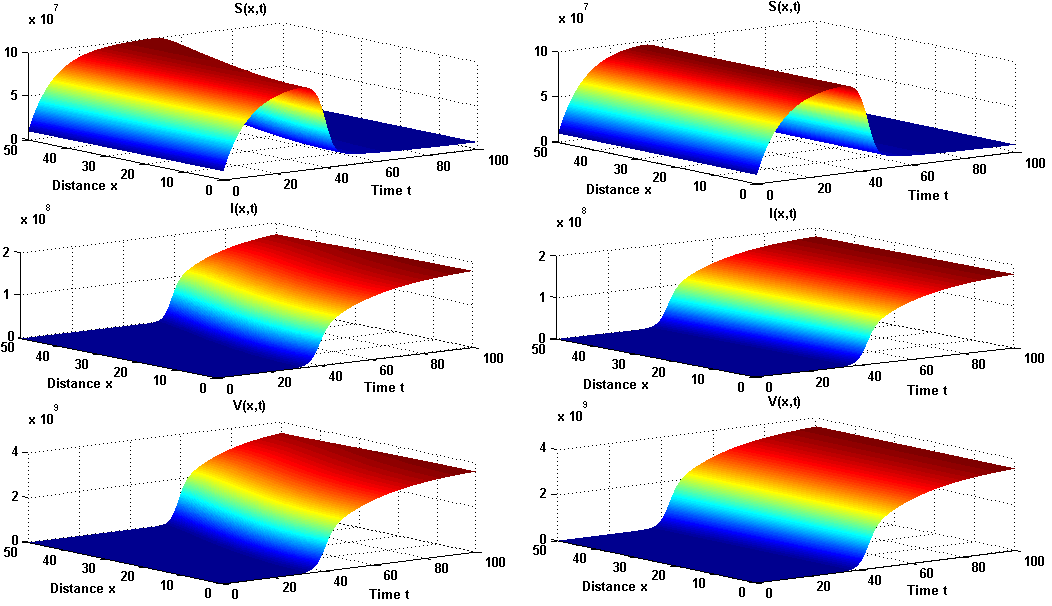}
\caption{Dynamics of system (\ref{b1}) under diffusion coefficients $D_1=D_2=D_3=1\ \text{mm}^2 \text{d}^{-1}$ (left) and $D_1=D_2=D_3=100\ \text{mm}^2 \text{d}^{-1}$ (right).}
 \label{fig3}
\end{figure}

\section{Conclusion}
In order to investigate the mechanism of viral infection and replication, we analyzed a diffusive intra host virus dynamics model that includes cell to cell transmission and allows for general nonlinear incidence functions. We established the well posedness of the model and studied the linear stability of its equilibria. The basic reproduction number $\mathcal{R}_{0}$ serves as a threshold that determines whether the infection dies out or persists. Specifically, we showed that the global dynamics is completely characterized by $\mathcal{R}_{0}$: if $\mathcal{R}_{0} \le 1$, then the disease free equilibrium $E_{0}$ is globally asymptotically stable, so the virus is cleared and the infection vanishes; if $\mathcal{R}_{0} > 1$, then the endemic equilibrium $E^{*}$ is globally asymptotically stable. From the expression of $\mathcal{R}_{0}$ it follows that ignoring either virus to cell transmission or cell to cell transmission may lead to an underestimation of the actual reproduction number. Hence, reducing only the contribution from virus to cell transmission may not be sufficient to eliminate the infection when cell to cell transmission is present. Our analysis also shows that, under homogeneous Neumann boundary conditions, the diffusion coefficients do not affect the global behavior of the system. Furthermore, by applying the nonstandard finite difference (NSFD) scheme, we derived a discrete version of the continuous model. We proved that this discretization preserves the key qualitative properties of the original system, including positivity, ultimate boundedness, and global stability of the equilibria, without imposing restrictions on the spatial or temporal step sizes.

The model considered here extends several earlier works \cite{r40, r41, r42, r43, r44, r45} and the results obtained here improve some known results. A natural continuation of this work is to incorporate a logistic growth term for uninfected target cells and to consider more general infection functions. This will be addressed in future studies.

\subsection*{Acknowledgments}
 The authors would like to thank the anonymous referee for his/her comments that helped us improve this article.
 
\subsection*{Data availability statement}
 The data that support the findings of this study are available within the article.
 
 \subsection*{Declaration of generative AI }
During the preparation of this work the author used ChatGPT and QuillBot were used to polish writing.

%
%
%
\end{document}